\documentclass[11pt,lefteqn]{article}

\usepackage{amsmath,amsfonts,amsthm,amssymb,graphics,epsfig,color}

\newcommand{\C}{\mathbb C}

\renewcommand{\Im}{\mathop {\rm Im}\nolimits}

\newcommand*{\clos}[1]{\overline{#1}}

\newtheorem{theorem}{Theorem}
\newtheorem{lemma}[theorem]{Lemma}

\begin{document}
\date {\today}
\title{Spectral sets: numerical range and beyond}

\author{Michel Crouzeix\footnote{Univ.\,Rennes, CNRS, IRMAR\,-\,UMR\,6625, F-35000 Rennes, France.
email: michel.crouzeix@univ-rennes1.fr},  Anne Greenbaum\footnote{University of Washington,
Applied Math Dept., Box 353925, Seattle, WA 98195.  email:  greenbau@uw.edu}
}

\maketitle

\begin{abstract} 
We extend the proof in [M.~Crouzeix and C.~Palencia, {\em The numerical range is a $(1 + \sqrt{2})$-spectral
set}, SIAM Jour.~Matrix Anal.~Appl., 38 (2017), pp.~649-655] to show that other regions in the
complex plane are $K$-spectral sets.  In particular, we show that various annular regions
are $(1 + \sqrt{2} )$-spectral sets and that a more general convex region with a circular hole 
or cutout is a $(3 + 2 \sqrt{3} )$-spectral set.
We demonstrate how these results can be used to give bounds on the 
convergence rate of the GMRES algorithm for solving linear systems and on that of rational 
Krylov subspace methods for approximating $f(A)b$, where $A$ is a square matrix, $b$ is a given vector, and $f$ is a function that can be uniformly approximated 
on such a region by rational functions with poles outside the region.
\end{abstract}

\paragraph{2000 Mathematical subject classifications\,:}47A25 ; 47A30

\noindent{\bf Keywords\,:}{ numerical range, spectral set}

\section{Introduction}
Let us consider a closed subset $X \subset\C$ of the complex plane and a bounded linear operator $A$ in a complex Hilbert space $(H, \langle, \rangle, \| \,\|)$. We will say that $X$ is a $K$-spectral set for 
$A$ if the spectrum of $A$ is contained in $X$ and if the following inequality
\begin{equation}
\label{eq1}
 \|f(A)\|\leq K\sup_{z\in X}|f(z)|,
\end{equation}
holds for all rational functions $f$ bounded in $X$. Note that $f(A)$ is naturally defined for such $f$ since, being bounded, $f$  has no pole in $X$. Let us denote by $\mathcal A(X)$ the set of
uniform limits in $X$ of  bounded rational functions; then,  by continuity, this inequality allows us to define $f(A)$ for $f\in \mathcal A (X)$ and inequality \eqref{eq1} still holds.

Now we consider a (non empty) bounded open subset $\Omega\subset \C$; we assume that its boundary $\partial \Omega$ is rectifiable and has a finite number of connected components. Then, if $A$ is a bounded linear operator with spectrum Sp$(A)$ contained in $\Omega$, it follows from the Cauchy formula that \eqref{eq1} holds with $X$ being the closure of $\Omega$ and $K=\frac{1}{2\pi }\int_{\partial \Omega}\|(\sigma I{-}A)^{-1}\|\,|d\sigma |$. 
But, this estimate is often very pessimistic, and we are looking for a better one. 
For that, we start with a rational function $f$ (bounded in $\Omega$) and we will consider the Cauchy formulae (for $z\in \Omega$)
\[
f(z)=\frac1{2\pi i}\int_{\partial \Omega}f(\sigma)\frac{d\sigma }{\sigma -z}, \quad f(A)=\frac1{2\pi i}\int_{\partial \Omega}f(\sigma)(\sigma I{-}A)^{-1}d\sigma .
\]
We will also introduce the Cauchy transforms of the complex conjugates of $f$
\begin{equation}\label{eq2}
g(z):=\frac1{2\pi i}\int_{\partial \Omega}\clos{f(\sigma)}\frac{d\sigma }{\sigma -z}, \quad g(A):=\frac1{2\pi i}\int_{\partial \Omega}\clos{f(\sigma)}(\sigma I{-}A)^{-1}d\sigma,
\end{equation}
and finally the transforms of $f$ by the double layer potential kernel
\begin{equation}\label{eq3}
s(z)=s(f,z):=\int_{\partial \Omega}f(\sigma(s))\mu (\sigma(s),z)\,ds \quad S=S(f,A):=\int_{\partial \Omega}f(\sigma (s))\mu (\sigma (s),A)\,ds.
\end{equation}
Here $s$ denotes the arc length of $\sigma =\sigma (s)$ on the (counter clockwise oriented) boundary
and $\mu $ the kernel\footnote{Note that $\mu $ is twice the usual kernel associated to the double layer potential.} given, for $ \sigma (s)\neq z$ , $z\in \clos{\Omega }$ and $\sigma (s)$ not in the spectrum of $A$, by
\begin{align}\label{eq4}
\mu (\sigma(s),z):=\frac1\pi \frac{d\arg(\sigma (s){-}z)}{ds}=\frac{1}{2\pi i}\Big(\frac{\sigma '(s)}{\sigma(s) -z}-\frac{\clos{\sigma '(s)}}{\clos{\sigma(s)} -\bar z}\Big),\\
\label{eq5}\mu (\sigma(s) ,A):=
\frac{1}{2\pi i}\big(\sigma '(s)(\sigma(s) I{-}A)^{-1}-\clos{\sigma '(s)}(\clos{\sigma(s)} I {-}A^*)^{-1}\big).
\end{align}
(Note that $\sigma $ is a Lipschitz function of $s$ with a constant $1$, thus $\sigma '(s)$ exists for almost every $s$.) 
From these definitions, it is clear that (for $z\in\Omega$)
\begin{equation}\label{eq6}
f(z)+\clos{g(z)}=s(z)\quad\text{and}\quad S^*=f(A)^*+g(A).
\end{equation}
Note also that, if we choose the constant function $f=1$, then $g=1$, $f(A)=g(A)=I$,
\[
\int_{\partial \Omega}\mu (\sigma ,z)\,ds=s(1,z)=2,\quad\text{if  }z\in\Omega\quad\text{and}\quad
\int_{\partial \Omega}\mu (\sigma ,A)\,ds=S(1,A)=2I.
\]

We will use the following lemma:
\begin{lemma}
\label{lem}
 Assume that $f$ is a rational function satisfying $|f|\leq 1$ in $\Omega$. Then  $g $ defined in \eqref{eq2} admits a continuous extension to ${\cal A}(\clos\Omega)$. Furthermore, if we set
\[
c_1:=\sup\{\max_{z\in \clos{\Omega}}| g(z) |\,: f\textrm{ rational function, } |f|\leq 1\textrm{ in }\Omega\},
\]
this constant satisfies
\[
c_1\leq \max_{\sigma _0\in \partial \Omega}\int_{\partial \Omega}|\mu (\sigma (s),\sigma _0)|\,ds.
\]
\end{lemma}
For the next theorem, we will assume that the set of uniform limits of rational functions bounded in $\Omega$ is the algebra
\[
\mathcal{A}(\clos{\Omega}):=\{ \, f\, : \,  f  \mbox{ is holomorphic in } \Omega \mbox{ and continuous in }\clos\Omega \, \}.
\]
This is automatically satisfied when $\mathbb{C}\backslash \Omega$ is connected,
since from Mergelyan's theorem the set of polynomial functions is then dense in $\mathcal{A}(\clos{\Omega})$. In the non simply connected case, this requires an assumption on the analytic capacity of the inner boundary curves; note that this condition is satisfied for smooth inner boundary curves \cite{vitu}.

\begin{theorem}\label{th}
Assume that  Sp$(A)\subset\Omega$ and that, for all rational functions $f$ satisfying $|f|\leq 1$ in $\Omega$, there exists $\gamma (f)\in \C$ such that $\|S(f,A){+}\gamma(f) I\|\leq 2\,c_2 $, with $c_2$ independent of $f$.
Then $\clos\Omega$ is a $K$-spectral set for the operator $A$ with a constant 
 \[
 K=c_2+\sqrt{c_2^2{+}c_1 {+}\hat\gamma  },\quad \text{with}\quad \hat\gamma  :=\max\{|\gamma(f)|\,: |f |\leq 1 \text{ in }\Omega\}.
 \]
\end{theorem}
One way to verify the hypotheses of Theorem \ref{th} is to introduce
\[
\lambda _{min}(\mu(\sigma ,A))=\inf\{\lambda \,: \lambda \in {\rm Sp}(\mu (\sigma ,A))\},
\] and use
\[\gamma =\gamma (f):=-\int_{\partial \Omega}f(\sigma (s))\lambda _{min}(\mu(\sigma(s) ,A))\,ds.
\]
Then we use the following result from \cite{cgl}
\begin{lemma}\label{cagrli}
 Assume that $f$ is a rational function satisfying $|f|\leq 1$ in $\Omega$ and that Sp$(A)\subset\Omega$, then it holds
 \[
  \|S(f,A)+\gamma(f) I\|\leq 2+\delta ,\quad\text{with}\quad \delta =-\int_{\partial \Omega}
  \lambda _{min}(\mu(\sigma(s) ,A))\,ds.
  \]
\end{lemma}
\noindent
Thus, the hypotheses of Theorem \ref{th} are satisfied with $c_2 = 1 + \delta / 2$ and 
$\hat{\gamma} = \int_{\partial \Omega} | \lambda_{min} ( \mu ( \sigma (s), A )) |\,ds$.

\medskip
 
The paper is organized as follows.  In section 2 we provide a proof of
Lemma 1; the less obvious part is the continuity which can be found in the literature 
but under stronger smoothness assumptions on the boundary and in a more general context, see for instance \cite{fo} or Carl Neumann \cite{neu} for the original proof.
Section 3 contains a proof of Theorem 2 using
a technique of Schwenninger \cite{rasc} that was also incorporated in \cite{cgl}
to improve upon the original result there.  Section 4 gives some estimates of 
$\lambda_{min}( \mu ( \sigma , A ) )$, and these are used in sections 5 and 6
to show that various annular-like regions and regions with circular cutouts are $K$-spectral sets and to bound the value of $K$.  Finally, in section 7, we show how these new theoretical 
results provide better bounds on the convergence rate of the GMRES algorithm for
solving linear systems and on that of the rational Arnoldi algorithm
for approximating $f(A)b$, where $A $ is a square matrix, $b$ is a given vector, and $f$  is a 
uniform limit of bounded rational functions on a region discussed in section 5 or 6. 
Note that some similar arguments have been used in \cite{BeckCrouz} for bounding Faber polynomials of an operator.

\section{Proof of Lemma \ref{lem}}
Recall that $f$ is a rational function bounded by 1 in $\Omega$ and that
\[
g(z):=\frac1{2\pi i}\int_{\partial \Omega}\clos{f(\sigma)}\frac{d\sigma }{\sigma -z},\qquad \text{for }z\in \Omega.
\]
Since $f$ is continuous on the boundary $\partial \Omega$, it is clear that $g$ is holomorphic in $\Omega$. It remains to show that $g$ has a continuous extension to $\clos\Omega$ and that it is bounded by $c_1$. For that, we will first remark that there exists a finite constant $\gamma _f$ such that 
\[
\gamma _f=\sup\{\frac{|f(z _1){-}f(z_2)|}{|z _1-z_2|}\,: z_1\neq z_2\in\clos\Omega\}.
\]
Indeed, there exist two sequences, $\{z_{1,n}\}$ $\{z_{2,n}\}$ in $\clos\Omega$ with $\gamma_f=\lim|f[z_{1,n},z_{2,n}]\,|$; after extraction of a subsequence if needed, we can assume that $z_{1,n}\to z_1\in\clos\Omega$ and $z_{2,n}\to z_2\in\clos\Omega$. This implies $\gamma _f=|f[z_1,z_2]\,|$ which is finite.
Now, we extend $g$ on the boundary by setting
\begin{equation}\label{eq7}
g(\sigma _0)=\int_{\partial \Omega}\big(\clos{f(\sigma (s))}{-}\clos{f(\sigma_0)}\big)\,\mu (\sigma (s),\sigma _0)\,ds+
\clos{f(\sigma _0)},\qquad \text{for }\sigma _0\in \partial \Omega.
\end{equation}
\begin{proof}[a) Proof of: $g$ is continuous in restriction to $\partial \Omega$] \phantom{toto}\ \\
Clearly, it suffices to show the continuity with respect to $\sigma _0$ of the integral part in \eqref{eq7}.
Note that $|f(\sigma (s)){-}f(\sigma_0)|\leq \gamma _f|\sigma (s){-}\sigma_0|$ and $|\mu (\sigma (s),\sigma_0)|\leq |\pi (\sigma (s){-}\sigma_0)|^{-1}$; the integrand being continuous for $s\neq t$ and bounded, the continuity of $g$ in restriction to $\partial \Omega$ follows from the dominated convergence theorem.
\end{proof}
\begin{proof}[b) Proof of: $g$ is continuous in  $\clos \Omega$]
It suffices to show that, if $z_n\to \sigma _0\in \partial \Omega$ with $z_n\in\Omega$, then $g(z_n)\to g(\sigma _0)$. For that we associate to each $z_n$ a point $\sigma _n\in \partial \Omega$ such that
\[
|z_n{-}\sigma _n|=\min\{|z_n{-}\sigma |\,: \sigma \in \partial \Omega\}.
\]
Clearly, it holds $\sigma_n\to \sigma _0$ whence, from the part a), it suffices to show that $g(z_n){-}g(\sigma _n)\to0$. Using that $\int_{\partial \Omega}\mu (\sigma(s) ,z)\,ds=2$ if $z\in \Omega$, we can write
\begin{align*}
g(z_n)-g(\sigma_n)&=\clos{s(z_n)}-\clos{f(z_n)}-g(\sigma_n)\\
&=\int_{\partial \Omega}\clos{\Delta f(\sigma (s),\sigma _n)}\ \Delta \mu (\sigma (s),z_n,\sigma _n )\,ds
+\clos{f(\sigma _n)}{-}\clos{f(z_n)},
\end{align*}
with
\[
\Delta f(\sigma(s),\sigma _n )=f(\sigma(s) ){-}f(\sigma _n),\quad\Delta \mu (\sigma(s) ,z_n,\sigma _n)=\mu (\sigma(s),z_n){-}\mu (\sigma(s) ,\sigma _n).
\]
Note that 
\[
|\Delta f(\sigma(s),\sigma _n )\Delta \mu (\sigma (s),z_n,\sigma _n )|\leq \frac{ \gamma _f\,|\sigma(s){-}\sigma _n| }{\pi }\Big|\Im \frac{\sigma '(s)(z_n-\sigma _n)}{(\sigma (s){-}z_n)(\sigma (s){-}\sigma _n) }\Big|\leq \frac{\gamma _f}{\pi}.
\]
Furthermore $\Delta f(\sigma(s),\sigma _n )\Delta \mu (\sigma (s),z_n,\sigma _n )\to 0$ for all $s$ such that $\sigma (s)\neq \sigma _0$. We deduce that $g(z_n){-}g(\sigma _n)$ tends to $0$ from the dominated convergence theorem.
\end{proof}
\begin{proof}[c) Proof of the bound]
We now remark that $\int_{\partial \Omega}|\mu (\sigma (s),\sigma _0)|\,ds=1$ for all points $\sigma _0\in \partial \Omega$ where $\sigma$ is differentiable; therefore we deduce from Equation\,\eqref{eq7}
\[
g(\sigma _0)=\int_{\partial \Omega}\clos{f(\sigma (s))}\,\mu (\sigma (s),\sigma _0)\,ds,\quad \textrm{ for almost every }\sigma _0\in \partial \Omega.
\]
This implies the bound $|g|\leq \max_{\sigma _0\in \partial \Omega}\int_{\partial \Omega}|\mu (\sigma (s),\sigma _0)|\,ds$ on the boundary and then in the interior by the maximum principle.
\end{proof}

\section{Proof of Theorem \ref{th}}
\begin{proof}[]
Let $A$ be a bounded operator satisfying Sp$(A)\subset \Omega$.
We set 
\[
K=K(A)=\sup\{\|f(A)\|\,: f\in\mathcal A(\Omega),\ Ê|f|\leq 1 \text {  in }\Omega\}.
 \]
 We have a first estimate $K\leq \frac{1}{2\pi }\int_{\partial \Omega}\|(\sigma I{-}A)^{-1}\|\,|d\sigma |$. 
We have seen in \eqref{eq6} that $S^*=f(A)^*+g(A)$, thus
 \[
( f(A)^*f(A))^2= f(A)^*f(A)\,(S{+}\gamma I)^*f(A)-f(A)^*(f(g+\bar\gamma )f)(A).
  \]
If we assume $|f|\leq 1$ in $\Omega$, then we can use the bounds $\|f(A)\|\leq K$ and, using that $g$ is a uniform limit of rational functions, $\|(f(g+\bar\gamma )f)(A)\|\leq K \sup_{\Omega}|f(g{+}\bar\gamma )f|\leq c_1{+}\hat\gamma $ to get
 \[
\|( f(A)^*f(A))^2\|\leq 2\,K^3c_2+K^2(c_1 {+}\hat\gamma) .
  \]
 Whence, for the supremum, $K^4\leq 2\,c_2K^3{+}(c_1{+}\hat \gamma  )K^2$, which shows that $K\leq c_2{+}\sqrt{c_2^2{+}c_1 {+}\hat\gamma }$.
\end{proof}

\section{Some estimates of $\lambda _{min}(\mu (\sigma ,A))$}
In this section, we fix a point $\sigma _0=\sigma (s_0)\in\partial \Omega$ with unit tangent
$\sigma '_0=\sigma '(s_0)$; the half-plane $\Pi _0=\{z\in\C\,:\Im\big(\sigma '_0(\clos{\sigma_0}{-}\bar z)\big)\geq 0\}$ has the same outward normal as $\Omega$ at $\sigma_0$.
Note that
\[
\mu (\sigma _0,A)=\frac1{2\pi i}\big(\sigma '_0(\sigma _0I{-}A)^{-1}-\clos{\sigma '_0}(\clos{\sigma} _0I{-}A^{*})^{-1}
\big)
\]
depends on $\sigma _0$ and $\sigma'_0$, but not on the other values of $\sigma (\cdot)$. 
\begin{lemma}
 Assume $W(A)\subset\Pi _0$; then  $\lambda _{min}(\mu (\sigma_0,A))\geq 0$. If furthermore $\sigma_0\in \partial W(A)$, then $\lambda _{min}(\mu (\sigma_0,A))=0$.
\end{lemma}\label{lem4}
\begin{proof}
 Let us consider $v\in H$, $v\neq 0$. We set $u=(\sigma_0I{-}A)^{-1}v$ and $\alpha=\|u\|$; then $z=\langle Au,u\rangle/\alpha^2\in W(A)$, whence
 \[
 \pi \langle \mu (\sigma_0,A)v,v\rangle=\Im\big(\sigma '_0\langle u,(\sigma_0I{-A)}u\rangle\big)=
 2\alpha^2\Im\big(\sigma '_0(\clos{\sigma_0}{-}\bar z)\big)\geq 0,
 \]
 which shows that $\lambda _{min}(\mu (\sigma_0) ,A))\geq 0$. If furthermore $\sigma_0=\langle Au_0,u_0\rangle\in W(A)$, $\|u_0\|=1$, then choosing $v=(\sigma_0I{-}A)u_0$, thus $z=\sigma_0$, we obtain $\langle \mu (\sigma_0,A)v,v\rangle=0$, whence  $\lambda _{min}(\mu (\sigma_0,A))= 0$.
 \end{proof}
 \noindent{\bf Remark.} {\it In particular, if $\Omega$ is a convex open set which contains $W(A)$, then 
 we deduce $\mu (\sigma ,A)\geq 0$ for all $\sigma \in\partial \Omega$; thus, $\|S(f,A)\|\leq \big\|\int_{\partial \Omega} \mu (\sigma ,A)\,ds\big\|=2$, for every $f$ with $|f| \leq 1$ in $\Omega$. Since $\Omega$ is convex, we deduce $\mu (\sigma ,\sigma_0)\geq 0$ and $c_1=\max_{\sigma _0\in \partial \Omega}\int_{\partial \Omega}\mu (\sigma (s),\sigma _0)\,ds=1$. Therefore, we deduce from Theorem \ref{th}, used with $\gamma (f)=0$ and $c_2=1$, that $\Omega$ is a $K$-spectral set for $A$, with $K\leq 1+\sqrt2$. In particular, using a decreasing sequence of convex $\Omega$ tending to $W(A)$, we refind the Palencia estimate: the numerical range is a $(1{+}\sqrt2)$-spectral set.} 
 
\begin{lemma}\label{lem5}
 Assume $|\sigma_0{-}\omega|=R$ and $\{z\in\C\,:|z{-}\omega|\leq R\}\subset \Pi _0$. If  $\|A{-}\omega I\|\leq R$, then  $\lambda _{min}(\mu (\sigma_0,A))\geq \frac{1}{2\pi R }$.
\end{lemma}
\begin{proof} Without loss of generality, we may assume $\omega=0$, $\sigma_0=R\,e^{i\theta }$,
$\sigma '_0=i\,e^{i\theta }$.  Then
\begin{align*}
2\pi  \mu (\sigma_0 ,A)-\frac{1}{R}I&=e^{i\theta }(\sigma_0 I{-}A)^{-1}+e^{-i\theta }(\bar\sigma_0 I{-}A^*)^{-1}-\frac{1}{R}I\\
&=\frac1{R}(\sigma_0 I{-}A)^{-1}(R^2I{-}AA^*)(\bar\sigma_0 I{-}A^*)^{-1}\geq 0,
\end{align*}
 since $\|A\|\leq R$.
\end{proof}
\noindent{\bf Remark.} {\it In particular, if $\Omega$ is the disk $\{z\in\C\,:|z{-}\omega|< R\}$,
 with the notations of Lemma\,\ref{cagrli}
we obtain $\delta \leq -1$. It is shown in \cite{cgl}[section 6.1] that in this case $K\leq \max(1,\|S{+}\gamma I )\|)$ whence, from this lemma, we get $K=1$. This is just the famous von Neumann inequality:
$\Omega$ is a spectral set for $A$.}
\medskip

\begin{lemma}\label{lem6}
Assume $\sigma_0{-}\omega=i\sigma '_0/R$ with $R>0$. If $\|(A{-}\omega I)^{-1}\|\leq R$, then  $\lambda _{min}(\mu (\sigma_0,A))\geq -\frac{R}{2\pi}$. 
\end{lemma}
\begin{proof} It suffices to consider the case $\omega=0$, $\sigma_0=r\,e^{-i\theta }$,
$\sigma'_0=-i\,e^{-i\theta }$, with $r=1/R$.  Then, with $B=A^{-1}$,
\begin{align*}
2\pi  \mu (\sigma_0 ,A)+RI&=-\big(e^{-i\theta }(\sigma_0 I{-}A)^{-1}+e^{i\theta }(\bar\sigma_0 I{-}A^*)^{-1}-RI\big)\\
&=-R(\sigma_0 I{-}A)^{-1}(r^2I{-}AA^*)(\bar\sigma_0 I{-}A^*)^{-1}
\\
&=r(\sigma_0 I{-}A)^{-1}A(R^2I{-}BB^*)A^*(\bar\sigma_0 I{-}A^*)^{-1}\geq 0,
\end{align*}
 since $\|B\|\leq R$.
\end{proof}

\noindent{\bf Remark.} {\it If $\sigma'_0$ denotes the unit tangent and if $\sigma'_0/i$ is the outward normal in a boundary point $\sigma_0$ of an open set $\Omega$, the assumption $\sigma_0{-}\omega=i\sigma '_0/R$ means that $\Omega$ and the exterior of disk $\{z\in\C\,:|z{-}\omega|^{-1}\leq R\}$ are tangent in $\sigma _0$ and have the same outward normal in this point.}\medskip

We now use $w(A)=\sup\{|\langle Av,v\rangle|\,: v\in H, \|v\|=1\}$ which is called the numerical radius of $A$.
\begin{lemma}\label{lem7}
Assume $\sigma_0{-}\omega=i\sigma '_0/R$ with $R>0$. If  $w((A{-}\omega I)^{-1})\leq R$, then  $\lambda _{min}(\mu (\sigma_0,A))\geq -\frac{R}{\pi}$.
\end{lemma}
\begin{proof} It suffices to consider the case $\omega=0$, $\sigma_0=r\,e^{-i\theta }$,
$\sigma '_0=-i\,e^{-i\theta }$, with $r=1/R$; we set $B=r\,A^{-1}$, then
\begin{align*}
2\pi  \mu (\sigma_0 ,A)+2RI&=R\big(2I {-}re^{-i\theta }(\sigma_0 I{-}A)^{-1}{-}re^{i\theta }(\bar\sigma_0 I{-}A^*)^{-1}\big)\\
&=R(\sigma_0 I{-}A)^{-1}(2\,AA^*{-}re^{-i\theta }A^*{-}re^{i\theta }A)(\bar\sigma_0 I{-}A^*)^{-1}
\\
&=R(\sigma_0 I{-}A)^{-1}A(2\,I{-}e^{-i\theta} B{-}e^{i\theta}B^*)A^*(\bar\sigma_0 I{-}A^*)^{-1}\geq 0,
\end{align*}
 since $w(B)\leq 1$.
\end{proof}

\section{Example 1: an annulus}
 We consider the annulus $\Omega=\mathcal A_R=\{z\in\C\,: r<|z|<R\}$, with $R>1$, $r=1/R$, and an invertible operator $A$ which satisfies $\|A\|< R$ and $\|A^{-1}\|<R$. 
Let us denote by $\Gamma _1=\{z\in\C\,: |z|=R\}$ and $\Gamma _2=\{z\in\C\,: |z|=r\}$ the two components of the boundary. It is clear that
\begin{align*}
&\text{if }\sigma _0\in \Gamma _1 \text{  then }\int_{\Gamma _1}|\mu (\sigma ,\sigma_0)|\,ds=1
\text{  and  } \int_{\Gamma _2}|\mu (\sigma ,\sigma_0)|\,ds=\frac4\pi \arcsin(r^2)<2,\\
&\text{if }\sigma _0\in \Gamma _2 \text{  then }\int_{\Gamma _2}|\mu (\sigma ,\sigma_0)|\,ds=1
\text{  and  } \int_{\Gamma _1}|\mu (\sigma ,\sigma_0)|\,ds=2,
\end{align*}
therefore $c_1\leq 3$. In fact, we can improve upon this estimate and show that $c_1 = 1$. 
\begin{lemma}\label{lem8}
For any rational function $f$ bounded by $1$ in ${\mathcal A}_R$, the associated function $g(z)$ defined by
\begin{equation}\label{defn_of_g}
g(z) = \frac{1}{2 \pi i} \int_{\partial {\mathcal A}_R} \overline{f( \sigma )} \frac{d \sigma}{\sigma - z} ,~~
z \in {\mathcal A}_R ,
\end{equation}
satisfies $| g(z) | \leq 1$ in ${\mathcal A}_R$.
\end{lemma}
\begin{proof}
Recall that $g$ has a continuous extension to the boundary given for $\sigma_0 \in \partial {\mathcal A}_R$ by
\[
g( \sigma_0 ) = \int_{\partial {\mathcal A}_R} \overline{f( \sigma )} \mu( \sigma , \sigma_0 )\,ds ,~~~
\mbox{ with } \mu ( \sigma (s) , z ) = \frac{1}{2 \pi i} \left( \frac{\sigma' (s)}{\sigma (s) -z} -
\frac{\overline{\sigma' (s)}}{\overline{\sigma(s)} - \bar{z}} \right) ,
\]
where $s$ denotes arclength on $\partial {\mathcal A}_R$.  Let $f_{\theta} (z) = f( z e^{i \theta} )$,
$g_{\theta} (z) = g( z e^{i \theta} )$, $\tilde{f} (z) = f(1/z)$, and $\tilde{g} (z) = g(1/z)$.
Then it is easily verified that if we replace $f$ by $f_{\theta}$ (resp.~by $\tilde{f}$), the 
associated function $g$ in (\ref{defn_of_g}) is replaced by $g_{\theta}$ (resp.~by $\tilde{g}$).
From this and the maximum principle, it suffices to show that $| f | \leq 1$ in ${\mathcal A}_R$
implies $| g(r) | \leq 1$.
Note that
\[
\mu ( \sigma , r) = - \frac{R}{2 \pi} ,~~\mbox{ if } \sigma = r e^{-i \theta} ,
\]
\[
\mu ( \sigma , r ) = \frac{R}{\pi} \frac{R^2 - \cos \theta}{R^4 - 2 R^2 \cos \theta + 1} ,~~\mbox{ if } 
\sigma = R e^{i \theta} .
\]
Let $\Gamma_R = \{ z \in \mathbb{C} : | z | = R \}$, $\Gamma_r = \{ z \in \mathbb{C} : |z| = r \}$.
On $\Gamma_r$, we write $\sigma (s) = r e^{-i \theta (s)}$, where $s = r \theta (s)$, $ds = r d \theta$,
$d \sigma = -i \sigma d \theta = -i \sigma R ds$.  Then
\[
\overline{g(r)} = \int_{\Gamma_R} f( \sigma ) \mu ( \sigma , r )\,ds - \frac{R}{2 \pi} 
\int_{\Gamma_r} f( \sigma )\,ds =
\int_{\Gamma_R} f( \sigma ) \mu ( \sigma , r )\,ds + \frac{1}{2 \pi i} \int_{\Gamma_r}
\frac{f( \sigma )}{\sigma}\,d \sigma .
\]
Using the fact that $\int_{\partial {\mathcal A}_R} \frac{f( \sigma )}{\sigma}\,d \sigma = 0$,
we obtain
\[
\overline{g(r)} = \int_{\Gamma_R} f( \sigma ) \mu ( \sigma , r )\,ds - \frac{1}{2 \pi i} \int_{\Gamma_R}
\frac{f( \sigma )}{\sigma}\,d \sigma =
\int_{\Gamma_R} f( \sigma ) \left( \mu ( \sigma , r ) - \frac{1}{2 \pi R} \right)\,ds .
\]
Finally, we remark that if $\sigma = R e^{i \theta}$, then $\mu ( \sigma , r ) - \frac{1}{2 \pi R} =
\frac{1}{2 \pi R} \frac{R^4 - 1}{R^4 - 2 R^2 \cos \theta + 1} > 0$, which implies
\[
| g(r) | \leq \int_{\Gamma_R} \left( \mu ( \sigma , r ) - \frac{1}{2 \pi R} \right)\,ds = 1 .
\]
\end{proof}

Now, we introduce the self-adjoint operator $\nu(\sigma ,A)=\mu (\sigma ,A){-}\frac1{2\pi i}\frac{\sigma '}{\sigma }I$. If $\sigma \in\Gamma _1$, we may write $\sigma =Re^{i\theta }$, $s=R\,\theta $, thus $\nu(\sigma ,A)=\mu (\sigma ,A)-\frac{1}{2\pi R}I\geq 0$ follows from Lemma\,\ref{lem5}. Similarly, if $\sigma \in\Gamma _2$, we may write $\sigma =r\,e^{-i\varphi  }$, $s=r\,\varphi $, thus $\nu(\sigma ,A)=\mu (\sigma ,A)+\frac{R}{2\pi }I\geq 0$ follows from Lemma\,\ref{lem6}.

Now, we consider a rational function  $f$ bounded by 1 in $\Omega$ and note that $\int_{\partial \Omega} f(\sigma )/\sigma \,d\sigma=0$.
Whence $\int_{\partial \Omega}f(\sigma )(\mu(\sigma ,A)-\nu (\sigma ,A))\,ds=\frac1{2\pi i}\int_{\partial \Omega}f(\sigma )/\sigma \,d\sigma =0$. We deduce
\[
\|S(f,A)\|=\Big\|\int_{\partial \Omega}f(\sigma )\nu (\sigma ,A)\,ds\Big\|
\leq  \Big\|\int_{\partial \Omega}\nu(\sigma ,A)\,ds\Big\|=\Big\|\int_{\partial \Omega}\mu(\sigma ,A)\,ds\Big\|=2. 
\]
We apply Theorem\,\ref{th} with $c_1 = 1$, $\gamma =0$, $c_2=1$, and obtain that $\mathcal A_R$ is a $K(R)$-spectral set for $A$ with some optimal constant $K(R)\leq 1{+}\sqrt{2}$.\medskip

\noindent{\bf Remark.} 
{\it If we assume only that
$\|A\|\leq R$ and $\|A^{-1}\|\leq R$ then, for all $R'>R$, $\mathcal A_{R'}$ is a $(1 + \sqrt{2})$-spectral set for $A$; taking the limit as $R'\to R$, we obtain that $\mathcal A_R$ is a $(1 + \sqrt{2})$-spectral set for $A$.
}

\noindent{\bf Remark.} {\it This bound $K(R)\leq 1{+}\sqrt{2}$ improves the previous one given in\,\cite{babecr}}
  \[
  K(R) \le \min\big( 2 +  \frac{R+1}{\sqrt{R^2+R+1}},\max(3,2+\sum_{n\geq 1}\frac{4}{1+R^{2n}})\big).
\]
{\it A first bound for this constant had been obtained by
 Shields\,\cite{shields} but this bound was unbounded for $R$ close to $1$. 
 Note also that a lower bound is known\,\cite{babecr}: $ K(R)\geq \gamma (R)$ with} 
\[
\gamma (R):=2(1{-}R^{-2})\prod_{n\geq 1}\Big(\frac{1-R^{-8n}}{1-R^{4-8n}}\Big)^2\leq 2.
\]
\medskip

 This result is still true if $\Omega$ is the intersection of two disks of the Riemann sphere.
\begin{theorem}
Let us consider  $\Omega=D_1\cap D_2$ with $D_1=\{ z \in\C\,: |z{-}\omega_1|< R_1\}$,
 $D_2=\{ z \in\C\,: |z{-}\omega_2|>1/ R_2\}$.
 If the operator $A$ satisfies $\|A{-}\omega_1I\|\leq  R_1$ and $\|(A{-}\omega_2I)^{-1}\|\leq  R_2$,
 then $\Omega$ is a $(1 + \sqrt{2})$-spectral set for $A$.
\end{theorem}
\begin{proof} We first consider the case where $\partial D_2\subset D_1$. Then,
 there exist $R$ and a Moebius function $\varphi(z)=\frac{az+b}{cz+d} $ such that $\varphi $ is one to one from $\Omega$ onto $\mathcal A_{R}$, from $D_1$ onto $\{z\in\C\,:|z|<R\}$, and
 from $D_2$ onto $\{z\in\C\,:|z|>1/R\}$; we set $B=\varphi (A)$. The von Neumann inequality shows that $\|B\|\leq R$ and $\|B^{-1}\|\leq R$; therefore $\mathcal A_R$ is a $(1 + \sqrt{2})$-spectral set for $B$, which is clearly equivalent to $\Omega$ is a $(1 + \sqrt{2} )$-spectral set for $A$.

We now consider the case where the intersection $\partial D_1\cap\partial D_2$ is two distinct points. Then, using $\varphi (z)=1/(z{-}c)$ with $c\in \partial D_2$, $c\notin D_1$,
 $D' _1=\varphi (D_1)$ is some disk $\{z\,: |z{-}\gamma _1 |<R'_1\}$ and $D' _2=\varphi (D_2)$ is a half-plane. From von Neumann, $B=\varphi (A)$ satisfies $\|B{-}\gamma _1I\|\leq R'_1$ and $W(B)\subset D'_2$; a fortiori, $W(B)\subset \Omega':=\clos{D'_1\cap D'_2}$.  Since $\Omega'$ is convex, $\Omega'$ is a $(1{+}\sqrt2)$-spectral set for $B$, thus $\Omega$ is a $(1{+}\sqrt2)$-spectral set for $A$.

 The case where the intersection $\partial D_1\cap\partial D_2$ is only one point
 follows from the case $\partial D_2\subset D_1$ by increasing $D_1$ in $D_1'$ and then letting $D'_1$ tend to $D_1$.
\end{proof}

\section{Example 2: another domain with a hole or cutout}

We now consider the case where $\Omega=\Omega_1\cap\Omega_2$ is the intersection of a bounded convex domain $\Omega_1$ with the exterior of a disk 
$\Omega_2=\{z\in\C\,: |z{-}\omega|^{-1}< R\}$.
Then, arguing as at the start of the previous section, it can be seen that
$\max_{\sigma _0}\int_{\partial \Omega}|\mu (\sigma ,\sigma_0)|\,ds=3$, therefore $c_1\leq 3$. 
 We now assume $\clos{W(A)}\subset \Omega_1$, $w((A{-}\omega I)^{-1})<R$ and that
either $\partial \Omega_2\subset\Omega_1$, or
 the number of intersection points of $\partial \Omega_1$ and $\partial \Omega_2$ is finite.

Let $f$ be a rational function bounded by 1 in $\Omega$.
We consider $\Gamma _1=\partial \Omega_1\cap\clos{\Omega_2}$ and $\Gamma _2=\partial \Omega_2\cap\clos{\Omega_1}$, then $\partial \Omega=\Gamma _1\cup\Gamma _2$. We write $S(f,A)=S_1{+}S_2{+}S_3$ with
\[
S_1=\int_{\Gamma _1} f(\sigma )\mu (\sigma ,A)\,ds,\quad
S_2=\int_{\Gamma_2 } f(\sigma )\nu (\sigma ,A)\,ds,\quad
S_3=-\frac{R}\pi\int_{\Gamma _2} f(\sigma )\,ds\,I,
\]
with $\nu(\sigma ,A)=\mu (\sigma ,A)+\frac{R}\pi I$.  
If $\sigma \in \partial \Omega _1$, it holds $\mu (\sigma ,A)\geq 0$ since $\clos{W(A)}\subset \Omega_1$. We deduce 
\[
\|S_1\|\leq \big\|\int_{\Gamma _1}\mu (\sigma ,A)\,ds\big\|\leq \big\|\int_{\partial \Omega _1}\mu (\sigma ,A)\,ds\big\|=2.
\]
If $\sigma \in \Gamma_2 $, Lemma\,\ref{lem7}
shows that $\nu (\sigma ,A)\geq 0$, hence 
\[
\|S_2\|\leq \big\|\int_{\Gamma_2}\nu (\sigma ,A)\,ds\big\|
\leq \big\|\int_{\partial \Omega_2}\nu (\sigma ,A)\,ds\big\|=\frac R\pi \int_{\partial \Omega_2}\,ds
=2,
\]
since $\int_{\partial \Omega_2}\mu (\sigma ,A)\,ds =0$.
It is clear that $\|S_3\|\leq 2$. 

Therefore $c_2\leq 3$; applying Theorem\,\ref{th} with $c_1 =3$, $\gamma =0$, $c_2=3$, we obtain that $\clos\Omega$ is a $(3{+}2\sqrt3)$-spectral set for $A$.\medskip

\noindent
{\bf Remark.} {\it In particular, we can apply this result to the annulus $\mathcal A_{R}$, but now with $c_1 =1$, $\gamma =0$, $c_2=3$: under the assumptions $w(A)\leq R$ and $w(A^{-1})\leq R$, the annulus is a $(3{+}\sqrt{10})$-spectral set for $A$.
It improves, for $1<R<1.8837$, the previous estimates;
a uniform bound was not known up to now. 
The previous estimates were based on the splitting
\[
f(z)=f_1(z){+}f_2(z)\quad \text{with}\quad  f_1(z)=\sum_{n\geq 0}a_nz^n,\   f_2(z)=\sum_{n<0}a_nz^n,
\] 
and an estimate of $\|f_1\|_{D_1}+\|f_2\|_{D_2}$. Here $D_1=\{z\in\C\,: |z|<R\}$, 
$D_2=\{z\in\C\,: |z|>R^{-1}\}$, $\|f\|_D=\sup\{|f(z)|\,:z\in D\}$. From our assumptions, $D_1$ and $D_2$ are 2-spectral sets for $A$, therefore $\|f(A)\|\leq \|f_1(A)\| {+}\|f_2(A)\|\leq 2(\|f_1\|_{D_1}+\|f_2\|_{D_2})$. Two estimates $\|f_1\|_{D_1}+\|f_2\|_{D_2}\leq \max(3, 2+\psi (R))$, with $\psi (R)=\sum_{n\geq 1}\frac{4}{R^{2n}-1}$ and 
$\|f_1\|_{D_1}+\|f_2\|_{D_2}\leq 2+\frac{1}{\pi }\int_0^{\pi}\Big|\frac{R^2+e^{i\theta}}{R^2-e^{i\theta}}\Big|\,d\theta$ follow from \cite[Lemme 2.1(a) and (b)]{crzx2}. Choosing the best established estimate in each case, the annulus is a $K(R)$-spectral set for $A$, with
\begin{align*}
K(R)&\leq  3{+}\sqrt{10}\simeq 6.1623\hskip1.5cm\text{ if  }1<R<1.8837,\\
K(R)&\leq  4+\frac{2}{\pi }\int_0^{\pi}\Big|\frac{R^2+e^{i\theta}}{R^2-e^{i\theta}}\Big|\,d\theta
\hskip.3cm\text{ if  }
1.8837<R<2.3639,\\
K(R)&\leq  4{+}\sum_{n\geq 1}\frac{8}{R^{2n}-1}\hskip1.5cm\text{ if  }2.3639<R<2.3912,\\
K(R)&\leq  6\hskip4cm\text{ if  }2.3912<R.
\end{align*}}

\section{Some Applications}
The $K$-spectral sets derived in the previous sections can be used to give bounds
on the norm of the residual in the GMRES algorithm for solving a nonsingular 
linear system $Ax=b$ or on the error in an approximation to $f(A)b$ generated
by the rational Arnoldi algorithm.

\subsection{GMRES}
The GMRES algorithm generates, at each step $k$, an approximate solution $x_k$
for which the 2-norm of the residual, $b - A x_k$, is minimized over a Krylov subspace; that is,
\[
\| r_k \| = \min \{ \| p_k (A) r_0 \| : p_k \in {\cal P}_k ,~p_k (0) = 1 \} ,
\]
where ${\cal P}_k$ is the set of polynomials of degree at most $k$.  A bound independent
of the initial residual $r_0$ is
\[
\frac{\| r_k \|}{\| r_0 \|} \leq \min \{ \| p_k (A) \| : p_k \in {\cal P}_k ,~p_k (0) = 1 \} .
\]
It follows from \cite{crpa} that if $0 \notin W(A)$ then
\begin{equation}
\frac{\| r_k \|}{\| r_0 \|} \leq (1 + \sqrt{2} ) \min \{ \max_{z \in W(A)} | p_k (z) | :
p_k \in {\cal P}_k ,~p_k (0) = 1 \} , \label{crpaGMRES}
\end{equation}
and one thus obtains a bound on the GMRES residual norm in terms of an approximation
problem in the complex plane:  How small can a $k$th degree polynomial with value $1$
at the origin be on $W(A)$.

If $W(A)$ contains the origin, however, the bound (\ref{crpaGMRES}) is not useful.
One way to avoid this problem
was devised in \cite{ChoiGreen}:  Note that if $B = A^{1/m}$, then for $m$ large enough
$W(B)$ will not contain the origin, nor will the set $W(B )^m := \{ z^m : z \in W(B) \}$.
If $\varphi (z) = z^m$, then it follows from \cite{crpa} that
\[
\| p_k (A) \| = \| p_k \circ \varphi (B) \| \leq (1 + \sqrt{2} ) \max \{ | p_k \circ 
\varphi (z) | : z \in W(B) \} = ( 1 + \sqrt{2} ) \max \{ | p_k ( \zeta ) | : \zeta \in
W(B )^m \} .
\]
Unfortunately, this bound requires knowledge of $W(B)$.

A region described in section 6, consisting of the intersection of $W(A)$ and the exterior of a
disk about the origin of radius $1/R$, where $R$ is the numerical radius of $A^{-1}$
may provide a better bound.  In Figure~\ref{fig:1} we plot this region for the
Grcar matrix\footnote{gallery('grcar',100) in MATLAB} of order $n=100$.  As shown
in section 6, this is a $(3 + 2 \sqrt{3})$-spectral set for $A$.  Also shown in
the figure is the set $\exp (W( \log (A) ))$, which was shown in \cite{ChoiGreen} 
to be $\lim_{m \rightarrow \infty} [ W( A^{1/m} ) ]^m$ and hence
(after applying the result in \cite{crpa}) to be a $(1 + \sqrt{2})$-spectral 
set for $A$.  

\begin{figure}[ht]
\centerline{\epsfig{file=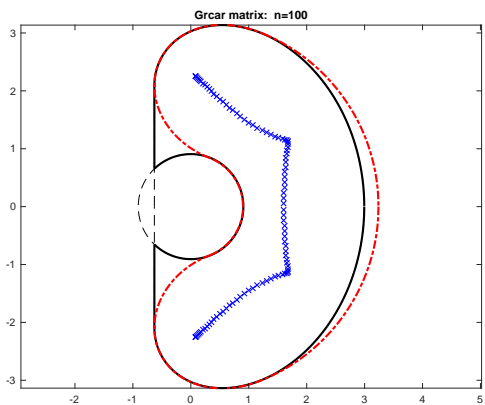,width=2.5in}}
\caption{Eigenvalues (x), boundary of $W(A)$ and circle about $0$ of radius $1/R$
where $R = w( A^{-1} )$ (thin dashed curves), and boundary of intersection of $W(A)$ with  
exterior of disk (thick solid curve).  This is a $(3 + 2 \sqrt{3})$-spectral set for $A$.
Also shown is the boundary of the set $\exp ( W( \log (A) ))$ (thick dash-dot curve), 
which was shown in \cite{ChoiGreen} (after applying \cite{crpa}) to be 
a $(1 + \sqrt{2})$-spectral set for $A$.}
\label{fig:1}
\end{figure}

\subsection{Rational Arnoldi algorithm}
Rational Krylov space methods (as well as standard Krylov space methods like the
Arnoldi algorithm) can be used to approximate the product of a function of a
matrix with a given vector: $f(A) b$.  The approximation at iteration $m$ is
of the form $r_m (A) b$, where $r_m = p_{m-1}/q_{m-1}$ is a rational function
with a prescribed denominator polynomial $q_{m-1} \in {\cal P}_{m-1}$.  The
rational Krylov space of order $m$ associated with $A$, $b$, and $q_{m-1}$
is defined as
\[
{\cal Q}_m (A,b) := [ q_{m-1} (A) ]^{-1} \mbox{span} \{ b, Ab, \ldots , A^{m-1} b \} .
\]
See, for example, \cite{Guttel} for an excellent review article.

Let $V_m \in \mathbb{C}^{n \times m}$ be an orthonormal basis for ${\cal Q}_m (A,b)$.
The rational Arnoldi approximation to $f(A) b$ from ${\cal Q}_m (A,b)$ is
\[
f_m^{\mbox{RA}} : = V_m f( A_m ) V_m^{*} b ,~~\mbox{ where }~ A_m := V_m^{*} A V_m .
\]
It is shown in \cite{Guttel} that if $f(A)b$ lies in the rational Krylov
subspace ${\cal Q}_m (A,b)$, then the rational Arnoldi approximation at step $m$ will be
exact: $f_m^{\mbox{RA}} = f(A) b$.  This is then used to show near-optimality of
the rational Arnoldi approximation to $f(A) b$.  Since $r_m (A) b = V_m r_m ( A_m ) V_m^{*} b$
for every rational function $r_m \in {\cal P}_{m-1}/ q_{m-1}$, we can write
\begin{eqnarray}
\| f(A) b - f_m^{\mbox{RA}} \| & = & \| f(A) b - V_m f( A_m ) V_m^{*} b - 
( r_m (A) b - V_m r_m ( A_m ) V_m^{*} b ) \| \nonumber \\
 & \leq & \| f(A) b - r_m (A) b \| + \| V_m ( f( A_m ) - r_m ( A_m ) ) V_m^{*} b \| \nonumber \\
 & \leq & ( \| f(A) - r_m (A) \| + \| f( A_m ) - r_m ( A_m ) \| ) \| b \| . \label{nearopt}
\end{eqnarray} 
Since $W( A_m ) \subset W(A)$, it follows, using the result in \cite{crpa}, that
$W(A)$ is a $(1 + \sqrt{2} )$-spectral set for both $A$ and $A_m$
and hence that
\begin{equation}
\| f(A) b - f_m^{\mbox{RA}} \| \leq 2 ( 1 + \sqrt{2} ) \| b \| 
\min \{ \max_{z \in W(A)} | f(z) - r_m (z) | : r_m \in {\cal P}_{m-1} / q_{m-1} \} . 
\label{bound_W(A)}
\end{equation}

To use the estimate (\ref{bound_W(A)}), the rational function $r_m$ should have no poles
in $W(A)$.  But if the function $f$ to be approximated has a pole in $W(A)$, then
it would be reasonable for $r_m$ to have one at the same point.  Here the
annulus of section 5, as well as the annulus or cutout region in section 6 might be 
useful for bounding the error in the approximation $f_m^{\mbox{RA}}$.  While
these regions are $K$-spectral sets for $A$, however, they might not be (with the same value of $K$)
for $A_m$.  The norm and numerical radius of $A_m$ are less than or equal to those of $A$,
but it is {\em not} guaranteed that the norm or numerical radius of $A_m^{-1}$
is less than or equal to that of $A^{-1}$.  Still, this is often the case,
and assuming that it is, one can use (\ref{nearopt}) to bound the error in
the rational Arnoldi approximation to $f(A)b$ in terms of the best uniform approximation 
to $f$ on one of these regions.

Taking $f(z) = 1/(1 - e^z )$ so that $f(A) = (I - e^A )^{-1}$, we used the
RKToolkit \cite{RKToolkit} to find a rational approximation to $f(A)b$ for
a random real vector $b$, again taking $A$ to be the Grcar matrix of size $n=100$.
We limited the number of poles to $m-1 = 5$ and ran routine rkfit to
find good pole placements for the rational Arnoldi algorithm.  As expected, it returned $0$
(which is inside $W(A)$) as one of the poles, and using the poles that it returned as the 
roots of $q_{m-1}$, we constructed an orthonormal basis $V_m$ 
for ${\cal Q}_m (A,b)$ and formed the rational Arnoldi approximation 
$f_m^{RA} = V_m f( A_m ) V_m^{*} b$.
The error $\| f(A) b - f_m^{RA} \| / \| b \|$ was about $2.7e-7$.
Evaluating the differences $| f(z) - \hat{r}_m (z) |$, where $\hat{r}_m$ is 
the rational function from routine rkfit, for $z$ in the annulus of section 5,
with outer radius $\| A \| \approx 3.2$ and inner radius $1/ \| A \|$, we found
the maximum difference to be about $9.4e-5$, leading to the upper bound
\[
\| f(A) b - f_m^{\mbox{RA}} \| / \| b \| \leq 2 ( 1 + \sqrt{2} )~9.4e-5 .
\]
The cutout region of section 6, which is the intersection of $W(A)$ with the exterior 
of a disk of radius $1/w( A^{-1} ) \approx 0.9$, provides a better bound.  
The maximum value of 
$| f(z) - \hat{r}_m (z) |$ on this set was about $2.7e-6$, 
leading to the error bound
\[
\| f(A) b - f_m^{\mbox{RA}} \| / \| b \| \leq 2 ( 3 + 2 \sqrt{3} ) ~2.7e-6 , 
\]
which must hold for every vector $b$ (provided that $A_m$ satisfies $w( A_m^{-1} ) \leq
w( A^{-1} )$, as it did in this case, so that this region is also a $(3 + 2 \sqrt{3} )$-spectral
set for $A_m$). A contour plot of $| f(z) - \hat{r}_m (z) |$
is shown in figure \ref{fig:2}(a), along with the annulus of section 5 and the cutout region
of section 6.  While $\hat{r}_m$ is {\em not} the best uniform approximation to $f$
on either of these regions, it is small enough to provide a reasonable upper bound
for the error in the rational Arnoldi approximation.

As another example, again taking $f(z) = 1/(1 - e^z )$, but now taking $A$ to be the 
matrix generated in MATLAB by typing `gallery('smoke',100)' (which is a 
$100$ by $100$ matrix with $1$'s on the superdiagonal, a $1$ in position $(100,1)$,
and powers of roots of unity along the diagonal), we ran routine rkfit
to find $m-1 = 5$ poles to use in a rational Arnoldi approximation to $f(A)b$, and it
again returned $0$ (which is inside $W(A)$) as one of the poles.  Using these poles in 
the rational Arnoldi
algorithm, with a random vector $b$, led to an error $\| f(A)b - f_m^{RA} \| / \| b \| 
\approx 2.0e-8$.  
While the difference $| f(z) - \hat{r}_m (z) |$ was large on the annulus of section 5, 
with outer radius $\max \{ \| A \| , \| A^{-1} \| \} \approx 31.8$ and inner radius
the reciprocal of this, it was small on the region of 
section 6 consisting of the intersection of $W(A)$ and the exterior of a disk
about $0$ of radius $1/w( A^{-1} ) \approx 0.04$.  Now this disk was a subset of $W(A)$,
so this was a different region with a hole in it.  The maximum value
of $|f(z) - \hat{r}_m (z)|$ on this region was about $7.9e-8$, leading to the bound
\[
\| f(A) b - f_m^{\mbox{RA}} \| / \| b \| \leq 2 ( 3 + 2 \sqrt{3} )~7.9e-8 ,
\]
which holds for all $b$ (again assuming that $w( A_m^{-1} ) \leq w( A^{-1} )$, as
it was in this case).
Figure \ref{fig:2}(b) shows a contour plot of $| f(z) - \hat{r}_m (z) |$.

\begin{figure}[ht]
\centerline{\epsfig{file=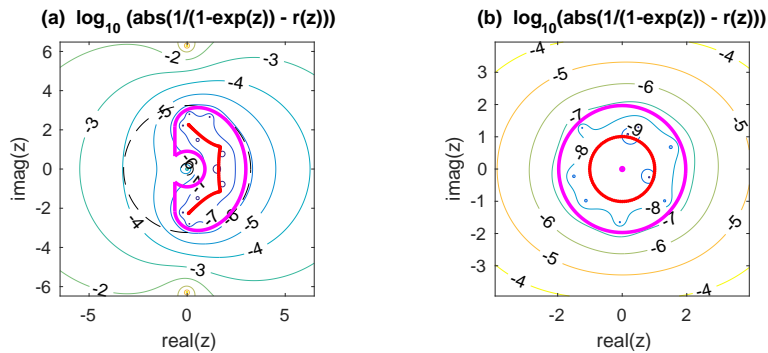,width=4in}}
\caption{Contour plot of $| f(z) - \hat{r}_m (z) |$ for (a) the
Grcar matrix and (b) the `smoke' matrix.  Dashed curve in (a) is
boundary of annulus with outer radius $\| A \|$, inner radius $1/ \| A \|$.
Thick solid curve in each plot is boundary of intersection of $W(A)$ with exterior of 
disk about $0$ of radius $1/ w( A^{-1} )$.  
Eigenvalues are marked with dots inside these regions.  (Note that in (b) the hole in
$W(A)$ is the tiny circle in the middle, while the dotted eigenvalues lie on the middle
circle.)
\label{fig:2}}
\end{figure}

\medskip
{\bf Acknowledgment:} The authors thank the American Institute of Mathematics (AIM) for hosting
a Workshop on Crouzeix's Conjecture at which many of the ideas in this paper originated.


\begin{thebibliography}{11}

\bibitem{babecr}{ \sc C.~Badea, B.~Beckermann, M.~Crouzeix}, {\em Intersections of several disks of the Riemann sphere as $K$-spectral sets}, Com. Pure Appl. Anal., 8 (2009), pp.~37--54.

\bibitem{BeckCrouz} {\sc B.~Beckermann and M.~Crouzeix}, 
{\em Faber polynomials of matrices for non-convex sets}, JAEN J.~of Approx., 6 (2014), pp.~219--231.


\bibitem{RKToolkit} {\sc M.~Berljafa, S.~Elsworth, and S.~G\"{u}ttel}, {\em Rational Krylov toolbox for MATLAB},
\verb+http://guettel.com/rktoolbox/+ .

\bibitem{cgl}{\sc T.~Caldwell, A.~Greenbaum, and K.~Li}, {\em Some extensions of the Crouzeix-Palencia result}, arxiv:1707.08603, SIAM J. Matrix Anal. Appl., to appear.

\bibitem{ChoiGreen}{\sc D.~Choi, A.~Greenbaum}, {\em Roots of matrices in the study of GMRES
convergence and Crouzeix's conjecture}, {SIAM J.~Matrix~Anal.~Appl.}, 36 (2015), pp.~289--301.

\bibitem{crzx2} {\sc M.~Crouzeix},
{\em The annulus as a K-spectral set}, 
 Com. Pure Appl. Anal., 11 (2012), pp.~2291--2303.

\bibitem{crpa} {\sc M.~Crouzeix, C.~Palencia},
{\em The numerical range is a $(1{+}\sqrt2)$-spectral set}, {SIAM J. Matrix Anal. Appl.}, 38 (2017), pp.~649--655.

\bibitem{fo} {\sc G.~B.~Folland},
{\em Introduction to Partial Differential Equations}, Princeton University Press, Princeton, NJ, 1995.

\bibitem{Guttel} {\sc S.~G\"{u}ttel}, {\em Rational Krylov approximation of matrix functions:
Numerical methods and optimal pole selection}, GAMM-Mitt., 36 (2013), pp.~8--31.

\bibitem{neu}{\sc C.~Neumann}, {\em Untersuchungen \"uber das logarithmische und Newton'sche Potential}, Math. Ann., 13 (1878), pp.~255--300.

\bibitem{rasc} {\sc T.~Ransford, F.~L.~Schwenninger}, {\em Remarks on the Crouzeix-Palencia proof that the numerical range is a $(1{+}\sqrt2)$-spectral set}, SIAM J. Matrix Anal. Appl., 39 (2018), pp.~342--345.

\bibitem{shields} {\sc A.L.~ Shields},
{\em Weighted shift operators and analytic function theory}, in :
\newblock Topics in operator theory, pp.~49--128. Math. Surveys, no.~13,
Amer. Math. Soc., Providence, R.I., 1974.

\bibitem{vitu}{\sc A.G.~Vitushkin}, {\em The analytic capacity of sets in problems of approximation theory}, Russ. Math. Surv., 22 (1967), pp.~139--200.
\end{thebibliography}
\end{document}